\documentclass[11pt,reqno]{amsart}
\usepackage[utf8]{inputenc}
\usepackage{amsmath, amsfonts, amsthm, amssymb}
\usepackage[margin=1in]{geometry}
\usepackage{mathtools}
\usepackage{stix}
\usepackage{float}
\usepackage{hyperref}

\usepackage{array}
\newcolumntype{C}[1]{>{\centering}m{#1}}

\newtagform{brackets}[\textbf]{$\langle$}{$\rangle$}
\usetagform{brackets}

\allowdisplaybreaks[4]

\newcommand{\gte}{\geqslant}
\newcommand{\lte}{\leqslant}
\newcommand{\presection}{\newpage\ \medskip}

\newtheorem{theorem}{Theorem}[section]
\newtheorem{lemma}[theorem]{Lemma}
\newtheorem{corollary}[theorem]{Corollary}
\newtheorem{prop}[theorem]{Proposition}

\newtheorem*{remark}{Remark}
\newtheorem*{definition}{Definition}
\newtheorem*{notation}{Notation}

\newcommand\fsk[2]{F_{#1,#2}}
\newcommand\mskn[3]{\mu_{#1,#2,#3}}
\newcommand\nfsk[2]{\tilde{F}_{#1,#2}}
\newcommand\qskn[3]{\mathcal Q_{#1,#2,#3}}
\newcommand\tqskn[3]{\tilde{\mathcal Q}_{#1,#2,#3}}

\newcommand\sms[2]{#1^{(#2)}}

\newcommand\mtop[2]{\genfrac{}{}{0pt}{}{\scriptstyle #1}{\scriptstyle #2}}

\newcommand\pp{\,.}
\newcommand\pc{\,,}

\newcommand\nn{\nonumber}
\newcommand\nc{\nonumber\\}

\DeclareMathOperator{\msu}{\uplus}

\newcommand\mbr{\mathbb R}
\newcommand\mbz{\mathbb Z}
\newcommand\mbc{\mathbb C}
\newcommand\mbq{\mathbb Q}
\newcommand\mbf{\mathbb F}

\newcommand\mca{\mathcal A}
\newcommand\mcp{\mathcal P}

\newcommand\mfp{\mathfrak p}
\newcommand\mfe{\mathfrak e}

\newcommand\Lsp{\mathfrak S}
\newcommand\lsp{\sigma}
\newcommand\Znk{Z_{n,k}}
\newcommand\Zkk{Z_{k,k}}

\newcommand{\stirling}[2]{\genfrac[]{0pt}{}{#1}{#2}}
\newcommand{\Stirling}[2]{\genfrac\{\}{0pt}{}{#1}{#2}}
\newcommand{\eulerian}[2]{\mathcal A_{#1,#2}}

\newcommand\vfs{\vphantom{\binom ns}}

\newcommand{\xxx}{\medskip\begin{center}*\quad *\quad *\end{center}\medskip}

\newcounter{mstep}

\begin{document}

\title[Moser Polynomials and Eulerian Numbers]{Moser Polynomials and Eulerian Numbers}
\author{Dmitri V.~Fomin}
\address{Boston, USA}
\email{fomin@hotmail.com}
\date{\today}
\begin{abstract}
In this article we investigate properties of the Moser polynomials which appear in various problems from algebraic combinatorics. For instance, they can be used to solve the Generalized Moser's Problem on multiset recovery: Can a collection (multiset) of $n$ numbers can be uniquely restored given the collection of its $s$-sums? We prove some explicit formulas showing relationships between Moser polynomials and such popular algebraic combinatorial sequences as Eulerian and Stirling numbers.
\end{abstract}

\maketitle

\presection
\section{Introduction}
\label{sec:Intro}

Let us give a few formal definitions and notations that we will use throughout this article. 

\begin{notation}
For any natural number $p$ we will denote by $x^{[p]}$ the \textbf{``falling power'' polynomial}\footnote{also often called ``falling factorial''} $x(x-1)\cdots(x-p+1)$. That is,
$$
x^{[p]} = \prod_{m=0}^{p-1}(x-m) = p! \binom{x}{p} \pp
$$
Another common notation for this polynomial that you often see in texts on combinatorics is \textbf{Pochhammer symbol} $(x)_p$.
\end{notation}

\begin{notation}
We will denote by $\eulerian{n}{m}$ the so-called Eulerian number---the number of permutations of order $n$ with exactly $m$ ascents; an ascent in permutation $\pi = \{\pi_1, \pi_2, \ldots, \pi_n\}$ is index $1 \lte j < n$ such that $\pi_j < \pi_{j+1}$. In some texts the notation $\genfrac\langle\rangle{0pt}{}{n}{\raisebox{2pt}{$\scriptstyle m$}}$ is used.

Clearly, $\eulerian{n}{m} = 0$ if $m < 0$ or $m \gte n$. Therefore, skipping the zeros, all the Eulerian numbers can be arranged in the triangular shape constituting the Eulerian triangle, where the $m$th number in the $n$th row is $\eulerian{n}{m}$.
\end{notation}

The polynomial with coefficients taken from $n$th row of this triangle is called \textit{Eulerian polynomial} and denoted as $\mca_n(x)$, that is, 
$$
\mca_n(x) = \sum_{m=0}^{n-1}\eulerian{n}{m}x^m \pp
$$

Below is the table with the first eight rows of the Eulerian triangle:
\begin{table}[H]
\begin{tabular}{|m{1cm}|C{1cm}|C{1cm}|C{1cm}|C{1cm}|C{1cm}|C{1cm}|C{1cm}|m{1cm}|}
\hline \vspace{1mm}
$n \backslash m$ &  0  & 1     & 2     & 3      & 4     & 5     & 6   & \quad 7  \\
\hline \vspace{1mm} 
1 & \textbf{1}   &       &       &        &       &       &     &   \\
\hline \vspace{1mm}
2 & \textbf{1}   & \textbf{1}     &       &        &       &       &     &   \\
\hline \vspace{1mm}
3 & \textbf{1}   & \textbf{4}   & \textbf{1}     &       &        &       &       &   \\
\hline \vspace{1mm}
4 & \textbf{1}   & \textbf{11}  & \textbf{11}    & \textbf{1}     &        &       &       &    \\
\hline \vspace{1mm}
5 & \textbf{1}   & \textbf{26}  & \textbf{66}    & \textbf{26}    & \textbf{1}      &       &       &    \\
\hline \vspace{1mm}
6 & \textbf{1}   & \textbf{57}  & \textbf{302}   & \textbf{302}   & \textbf{57}     & \textbf{1}     &       &    \\
\hline \vspace{1mm}
7 & \textbf{1}   & \textbf{120} & \textbf{1191}  & \textbf{2416}  & \textbf{1191}   & \textbf{120}   & \textbf{1}     &    \\
\hline \vspace{1mm}
8 & \textbf{1}   & \textbf{247} & \textbf{4293}  & \textbf{15619} & \textbf{15619}  & \textbf{4293}  & \textbf{247}   & \quad \textbf{1}  \\
\hline
\end{tabular}
\end{table}

Various properties and formulas for Eulerian numbers and polynomials can be found in \cite{Pet}.

\begin{definition}
For any natural numbers $k \lte n$ the following two symmetric polynomials---$\mfp_{k,n}$ and $\mfe_{k,n}$---in $n$ variables $x_1$, \ldots, $x_n$, are defined by the formulas
\begin{align}
\mfp_{k,n}(x_1,\ldots,x_n) &= x_1^k + \ldots + x_n^k =  \sum_{i=1}^n x_i^k \ ; \nc
\mfe_{k,n}(x_1,\ldots,x_n) &= \sum_{1\lte \alpha_1 < \ldots < \alpha_k \lte n}x_{\alpha_1}x_{\alpha_2}\ldots x_{\alpha_k} 
\pp \nn
\end{align}
 
Both are, obviously, homogeneous symmetric polynomials of degree $k$. They are called, respectively, a \textbf{power-sum} polynomial and an \textbf{elementary symmetric} polynomial of $k$th order. When the set of variables is fixed, we will often denote these polynomials simply by $\mfp_k$ and $\mfe_k$.
\end{definition}

It is well-known that for any subfield $\mbf$ of complex numbers (such as $\mbq$, $\mbr$, or $\mbc$) both sets of polynomials $\{\mfe_k\}$ and $\{\mfp_k\}$ constitute a basis in the ring $\Lambda_\mbf[x_1,\ldots, x_n]$ of symmetric polynomials in $n$ variables.

\begin{definition}
For any natural number $k$ and any $n$-multiset $A = \{a_1,\ldots,a_n\}$ we define $\mfp_k(A)$, the power-sum of $k$th order of multiset $A$, as $\mfp_{k,n}(a_1,\ldots,a_n)$.
\end{definition}

\begin{notation}
For any two natural numbers $n$ and $s$ such that $n\gte s$, and an arbitrary $n$-multiset $A = \{a_1,\ldots,a_n\}$ we define the multiset $\sms As$ of its $s$-sums, i.e., the collection of all sums of the form
$$
a_{i_1}+a_{i_2}+\ldots+a_{i_s}\pc
$$
where $1\lte i_1 < i_2 < \ldots < i_s \lte n$.
\end{notation}

Obviously, power-sum $\mfp_k(\sms As)$ of multiset $\sms As$ is a symmetric homogeneous polynomial of degree $k$ in $a_1$, \ldots, $a_n$. Therefore if $k \lte n$, then this power-sum can be uniquely represented as
\begin{align}
\label{eq:fsk_sigma}
    \mfp_k(\sms As)
    &=  \qskn skn(\mfp_1(A), \mfp_2(A), \ldots, \mfp_{k-1}(A), \mfp_{k}(A)) \nc
    &=  \mskn skn\, \mfp_k(A) + \tqskn skn(\mfp_1(A), \mfp_2(A), \ldots, \mfp_{k-1}(A)) \pc
\end{align}
where $\qskn skn$ and $\tqskn skn$ are polynomials in variables $\mfp_1$, \ldots,$\mfp_{k-1}$, $\mfp_k$, and $\mfp_1$, \ldots, $\mfp_{k-1}$ respectively, and the coefficient $\mskn skn$ is a constant (in terms of variables $a_i$) which depends only on $s$, $k$, and $n$.

Representation of symmetric polynomials of $\sms As$ via $\mfp_i(A)$ is of interest not only for purely algebraic or combinatorial reasons. One well-known example from topology is computation of Chern classes for exterior powers of a vector bundle. Let $\omega:E\rightarrow M$ be a vector bundle of rank $n$; consider computation of Chern classes of its $s$th exterior power $\bigwedge^s\omega$ via Chern classes $c_i(\omega)$ of the original bundle. The result will be the formula (very similar to \eqref{eq:fsk_sigma}) which expresses elementary symmetric polynomials $\mfe_k(\sms As)$ of multiset $\sms As$ via elementary symmetric polynomials $\mfe_k(A)$ of multiset $A$.

While that formula is clearly different from \eqref{eq:fsk_sigma}, the ``top'' coefficient (at $c_k(\omega)$) in that formula is the same number $\mskn skn$ from Equation \eqref{eq:fsk_sigma}. Namely, we have
$$
c_k(\bigwedge^s\omega) = \mskn skn\, c_k(\omega) + \tilde{\mathcal R}_{s,k,n}(c_1(\omega), c_2(\omega), \ldots, c_{k-1}(\omega)) \pp
$$
(see the proof below in Proposition \ref{thm:eskn_coeff}.)

We will also show how formula \eqref{eq:fsk_sigma} and polynomials $\qskn skn$ can be used to solve the so-called (Generalized) Moser Problem, or the Multiset Recovery Problem. The Moser Problem asks whether, given the multiset $\sms As$, it is always possible to uniquely restore (recover) the original multiset $A$. This question was originally posed by Leo Moser in 1957 as a problem in \textsc{American Mathematical Monthly} for $s=2$ and $n=4, 5$ (see \cite{Mos}.)

In article \cite{Fom} the reader can find a comprehensive survey of results and methods on this problem, circa 2017. In the next two sections we will compute $\mskn skn$ and show how it can be used in the Moser Problem.

\presection
\section{Explicit formula for \texorpdfstring{$\qskn skn$}{Q s,k,n}}
\label{sec:Qformula}

In this section we prove the explicit formula for polynomials $\qskn skn$ and present some of its corollaries. 

Consider an integer partition $\lambda$ of $k$, that is, $\lambda = \{\lambda_1, \ldots, \lambda_d\}$, where $\{\lambda_i\}$ is the non-increasing sequence of $d$ positive integers such that their sum equals $k$. Then $\mfp_\lambda$ will denote monomial $\mfp_{\lambda_1}\cdots \mfp_{\lambda_d}$ in variables $\mfp_i$.

Let $\delta = \{\delta_1, \ldots, \delta_q\}$ denote the sequence of of partition $\lambda$'s multiplicities---meaning that there are exactly $q$ different numbers among $\lambda_i$ with $i$th of these numbers occurring $\delta_i$ times. Obviously, the sum of these multiplicities equals $d$. 

Since $\qskn skn$ is a polynomial in $\mfp_i$, it can be uniquely written in the following form
$$
\qskn skn = \sum_{\lambda\in\mcp(k)} c_{\lambda}\mfp_{\lambda} \pc
$$
with rational coefficients $c_\lambda$, where $\mcp(k)$ is the set of integer partitions of $k$. We cannot immediately claim that these coefficients are integers, as would be the case with elementary symmetric polynomials ($\mfe_k$ form a $\mbz$-basis of the ring of symmetric polynomials with integer coefficients $\Lambda_{\mbz}$, while $\mfp_k$ do not.)

\medskip

Our main objective now is to find an explicit formula for coefficient $c_\lambda$.

\begin{theorem}
\label{thm:qskn_coeff}
\begin{equation}
\label{eq:clambda}
c_{\lambda} =
\frac{(-1)^{s+d}k!}{\lambda_1!\cdots \lambda_d!\delta_1!\cdots \delta_q!}
    \sum_{p=0}^s
    \sum_{\mtop{m_1,\ldots,m_d\gte 1}{m_1+\cdots+m_d = s-p}}
    (-1)^p \binom np 
     m_1^{\lambda_1-1} \cdots m_d^{\lambda_d-1}
    \pc
\end{equation}
where the second summation is done over all length $d$ compositions of $s-p$; that is, the sequences of $d$ positive integers $\{m_1,\ldots,m_d\}$ such that the sum of these numbers equals $s-p$.
\end{theorem}

\begin{proof} 
Given an arbitrary multiset $A = \{a_1, \ldots, a_n\}$ of $n$ (complex or rational) numbers $a_i$, let us consider the two functions
\begin{equation}
\label{eq:fx_sum}
f(x) = \sum_{i=1}^n e^{a_i x}
\end{equation}
and
\begin{equation}
\label{eq:gxy_sum}
g(x, y) = \prod_{i=1}^n (1 + ye^{a_i x}) \pp
\end{equation}

Now, using Taylor series expansion
$$
e^{ax} = \sum_{j=0}^\infty \frac{a^j x^j}{j!} \pc
$$
we can rewrite formula \eqref{eq:fx_sum} as follows
$$
f(x) = \sum_{i=1}^n \sum_{j=0}^\infty \frac{a_i^j x^j}{j!}
     = \sum_{j=0}^\infty x^j \left(\frac{1}{j!}\sum_{i=1}^n a_i^j\right)
     = \sum_{j=0}^\infty x^j \frac{\mfp_j(A)}{j!} \pc
$$
and formula \eqref{eq:gxy_sum} as
\begin{align}
g(x, y) &= \sum_{m=0}^n y^m \sum_{i_1 < i_2 < \ldots < i_m} e^{a_{i_1}x}e^{a_{i_2}x}\ldots e^{a_{i_m}x} \nc
        &= \sum_{m=0}^n y^m \sum_{i_1 < i_2 < \ldots < i_m} e^{(a_{i_1} + a_{i_2} + \ldots + a_{i_m})x} \nc
        &= \sum_{m=0}^n y^m \sum_{i_1 < i_2 < \ldots < i_m} \sum_{j=0}^\infty \frac{\left(a_{i_1} + a_{i_2} + \ldots + a_{i_m}\right)^j x^j}{j!} \nc
        &= \sum_{m=0}^n y^m \sum_{j=0}^\infty x^j \sum_{i_1 < i_2 < \ldots < i_m} \frac{\left(a_{i_1} + a_{i_2} + \ldots + a_{i_m}\right)^j}{j!} \nc
        &= \sum_{m=0}^n \sum_{j=0}^\infty y^m x^j \frac{\mfp_j(\sms Am)}{j!}
         = \sum_{j,m\gte 0} x^j y^m \frac{\mfp_j(\sms Am)}{j!} \nn 
\end{align}
where $\mfp_j(\sms Am)$ is defined as zero if $m = 0$, except for $\mfp_0(\sms A0) = 1$. Also, obviously, $\mfp_j(\sms Am) = 0$ if $m > n$.

For convenience sake we will use---in this proof only---the following notations.
$$
\lsp_j     = \frac{\mfp_j(A)}{j!} \pc \quad 
\Lsp_{j,m} = \frac{\mfp_j(\sms Am)}{j!} \pp
$$

\noindent Clearly, $f(x)$ and $g(x,y)$ are generating functions for sequences $\lsp_j$ and $\Lsp_{j,m}$. The following formula ties these two functions together.

\begin{lemma}
\begin{equation}
\label{eq:gxy_expfkt}
g(x,y)=\exp\left(\sum_{j=1}^\infty (-1)^{j-1}f(j x)\frac{y^j}j \right)\pp
\end{equation}
\end{lemma}

\begin{proof}
\begin{align}
\frac d{dy}\ln g(x,y) &= \frac d{dy}\ln \prod_{i=1}^n\left(1+ye^{a_i x}\right) = \sum_{i=1}^n\frac1y\;
\frac {ye^{a_i x}}{1+ye^{a_i x}} = \nc
&= \sum_{i=1}^n \frac1y \sum_{j=1}^\infty (-1)^{j-1}\left(ye^{a_i x}\right)^j=
\sum_{j=1}^\infty (-y)^{j-1}\sum_{i=1}^n e^{j a_i x}=\nc
&= \sum_{j=1}^\infty (-y)^{j-1}f(j x)\pp \nn
\end{align}
Since $g(x,0) = 1 = \exp(0)$, we have
$$
g(x,y) = \exp\left(\int\limits_0^y \sum_{j=1}^\infty (-y)^{j-1}f(j x) \,dy \right) 
       = \exp\left(\sum_{j=1}^\infty (-1)^{j-1}f(j x)\frac{y^j}j \right)\pp
$$
\end{proof}

\begin{lemma}
\begin{equation}
\label{eq:pkas_pka}
\sum_{j,m \gte 0}\Lsp_{j,m} x^j y^m = (1+y)^n\exp
\left(\sum_{j,m \gte 1}(-1)^{m-1}\lsp_j m^{j-1} x^j y^m \right)\pp
\end{equation}
\end{lemma}
\begin{proof}
Using equation \eqref{eq:gxy_expfkt} as well as formulas
$$
g(x,y)=\sum_{j,m \gte 0}\Lsp_{j,m} x^j y^m \ ; \quad
f(x)=\sum_{j=0}^\infty \lsp_j x^j =n + \sum_{j=1}^\infty \lsp_j x^j \pc
$$
we obtain
\begin{align}
\sum_{j,m \gte 0} \Lsp_{j,m} x^j y^m 
&= \exp\left(\sum_{m=1}^\infty(-1)^{m-1}
\left(n + \sum_{j=1}^\infty \lsp_j(mx)^j\right) \frac{y^m}m\right) \nc
&= \exp\left (\sum_{m=1}^\infty (-1)^{m-1}n\frac{y^m}m\right)\cdot\exp
 \left( \sum_{m=1}^\infty (-1)^{m-1}\sum_{j=1}^\infty \lsp_j(m x)^j
 \frac{y^m}m\right) \nc
&= \left(\exp\left( \sum_{m=1}^\infty (-1)^{m-1}\frac{y^m}m \right)\right)^n
\cdot \exp\left(\sum_{j,m \gte 1}(-1)^{m-1}\lsp_j m^{j-1} x^j y^m
\right)\pp \nn
\end{align}
Now it suffices to note that
$$
\exp\left(\sum_{m=1}^\infty (-1)^{m-1}\frac{y^m}m\right) = \exp\bigl(\ln(1+y)\bigr)=1+y\pp
$$
\end{proof}

Let us go back to equation \eqref{eq:pkas_pka} and consider coefficients at the term $x^k y^s$ on both sides of it. Since they must be the same, the following equality holds.
\begin{equation}
\label{eq:lsp_binom}
\Lsp_{k,s} = \sum_{p=0}^s \binom{n}{p} c_{k,s-p} \pc
\end{equation}
where $c_{k,s-p}$ is the coefficient at the term $x^k y^{s-p}$ in 
\begin{equation}
\label{eq:c_ksp}
\exp\left(\sum_{j,m \gte 1}(-1)^{m-1}\lsp_j m^{j-1} x^j y^m \right) \pp
\end{equation}

Now in order to express coefficient $c_{k,s-p}$ through $\lsp_k$ we rewrite and expand the expression \eqref{eq:c_ksp} as follows.
\begin{align}
\exp\left(\sum_{j,m \gte 1}(-1)^{m-1}\lsp_j m^{j-1} x^j y^m \right) = 
1 &+ \frac{1}{1!}\left(\sum_{j,m \gte 1}(-1)^{m-1}\lsp_j m^{j-1} x^j y^m \right)  \nc
  &+ \frac{1}{2!} \left(\sum_{j,m \gte 1}(-1)^{m-1}\lsp_j m^{j-1} x^j y^m \right)^2 \nc
  &+ \frac{1}{3!} \left(\sum_{j,m \gte 1}(-1)^{m-1}\lsp_j m^{j-1} x^j y^m \right)^3 + \ldots \nn
\end{align}

Thus the coefficient at $x^ky^{s-p}$ equals
$$
\sum_{d=1}^{s-p}
    \frac{(-1)^{s-p-d}}{d!}
    \sum_{\mtop{J, M}{\|J\| = k, \|M\| = s-p}}
    \prod_{i=1}^d \lsp_{j_i}m_i^{j_i-1} \pc
$$
where summation is done over all length $d$ compositions $J = \{j_1, \ldots, j_d\}$ of number $k$, and all length $d$ compositions $M = \{m_1, \ldots, m_d\}$ of number $s-p$.

Finally, from this, using the formula \eqref{eq:lsp_binom} and definitions of $\Lsp_{j,m}$ and $\lsp_j$ we obtain
$$
\qskn skn(\mfp_1,\ldots,\mfp_k) = \sum_{p=0}^s \binom np
          \sum_{d=1}^{s-p} (-1)^{s-p-d}\frac1{d!}
          \sum_{\mtop{J, M}{\|J\| = k, \|M\| = s-p}}
          \frac{k!}{j_1!\cdots j_d!} 
          \prod_{i=1}^d m_i^{j_i-1} \mfp_{j_i}
          \pp
$$

If we combine the like terms, then each term that contains monomial $\mfp_\lambda = \mfp_{\lambda_1}\mfp_{\lambda_2}\cdots\mfp_{\lambda_d}$ occurs in the sum above exactly $\binom{d}{\delta_1,\ldots,\delta_q}$ times; therefore the coefficient $c_\lambda$ at $\mfp_\lambda$ equals
\begin{align}
\frac{k!}{\lambda_1!\cdots \lambda_d!}
\frac{d!}{\delta_1!\cdots \delta_q!}
&\sum_{p=0}^s \binom np
          (-1)^{s-p-d}\frac1{d!}
          \sum_{\mtop{m_1,\ldots,m_d\gte 1}{m_1+\cdots+m_d = s-p}}
          \prod_{i=1}^d m_i^{\lambda_i-1}
\nc
&=
\frac{(-1)^{s+d}k!}{\lambda_1!\cdots \lambda_d!\delta_1!\cdots \delta_q!}
    \sum_{p=0}^s 
    \sum_{\mtop{m_1,\ldots,m_d\gte 1}{m_1+\cdots+m_d = s-p}}
    (-1)^p \binom np 
    m_1^{\lambda_1-1} \cdots m_d^{\lambda_d-1}
\pc \nn
\end{align}
which concludes the proof.
\end{proof}

\begin{remark}
Here is another, slightly different, way to present the same expression.
$$
c_\lambda =
    \frac{(-1)^{d}k!}{\lambda_1!\cdots \lambda_d!\delta_1!\cdots \delta_q!}
    \sum_{\|M\| \lte s}
    (-1)^{\|M\|} \binom n{s-\|M\|} 
    m_1^{\lambda_1-1} \cdots m_d^{\lambda_d-1}
    \pc
$$
where summation is done over all length $d$ compositions $M = \{m_1, \ldots, m_d\}$ of a positive integer not greater than $s$.
\end{remark}

The following are immediate corollaries of the formula \eqref{eq:clambda}.

\begin{corollary}
\label{thm:qskn_int}
All coefficients $c_\lambda$ of polynomial $\qskn skn$ are integers.
\end{corollary}

\begin{proof}
It is sufficient to show that the coefficient 
$$
\frac{k!}{\lambda_1!\cdots \lambda_d!\delta_1!\cdots \delta_q!}
$$
on the right-hand side of the formula \eqref{eq:clambda} is an integer. But this is the number of ways to dissect the set with $k$ elements into $d$ subsets containing $\lambda_1$, $\lambda_1$, \ldots, and $\lambda_d$ elements, and therefore we are done.
\end{proof}


\begin{corollary}
\label{thm:qskn_zero}
Coefficient $c_\lambda$ is nonzero only if length $d$ of partition $\lambda$ does not exceed $s$.
\end{corollary}

\begin{proof}
Indeed, from formula \eqref{eq:clambda} it follows that $d \lte s-p \lte s$.
\end{proof}

\medskip

Finally, a short and easy proof of the fact we have mentioned at the end of Introduction section.

\begin{prop}
\label{thm:eskn_coeff}
$$
\mfe_k(\sms As) = \mskn skn \, \mfe_k(A) + \mathcal W(\mfe_1(A), \ldots, \mfe_{k-1}(A)) \pc
$$
where $\mathcal W(e_1, \ldots, e_{k-1})$ is some integer polynomial in $e_i$.
\end{prop}

\begin{proof}
From Newton-Girard identities (see Theorems 2.9--2.14 in \cite{MenRem}) we have 
\begin{align}
\mfe_k(a_1, \ldots, a_n) = \frac{(-1)^{k-1}}{k}\mfp_k(a_1, \ldots, a_n) + \mathcal E(\mfp_1, \ldots, \mfp_{k-1}) \pc
\nc
\mfp_k(a_1, \ldots, a_n) = (-1)^{k-1}k\mfe_k(a_1, \ldots, a_n) + \mathcal P(\mfe_1, \ldots, \mfe_{k-1}) \pc
\nn
\end{align}
where $\mathcal E$, $\mathcal P$ are polynomials with rational coefficients. Notice that the coefficients $\frac{(-1)^{k-1}}{k}$ and $(-1)^{k-1}k$ are dependent only on $k$, not on $n$. Thus we have
\begin{align}
\mfe_k(\sms As) &= \frac{(-1)^{k-1}}{k}\mfp_k(\sms As) + \mathcal E (\mfp_1(\sms As), \ldots, \mfp_{k-1}(\sms As)) \nc
&= \frac{(-1)^{k-1}}{k}\mskn skn\mfp_k(A) + \mathcal V(\mfp_1(A), \ldots, \mfp_{k-1}(A)) \nc
&= \mskn skn \mfe_k(A) + \mathcal W(\mfe_1(A), \ldots, \mfe_{k-1}(A)) \pc \nn
\end{align}
using the fact that $\mfp_i(\sms As)$ ($i = 1, \ldots, k-1$) are polynomials in $\mfp_1(A)$, \ldots, $\mfp_{k-1}(A)$. 
\end{proof}

\presection
\section{Moser polynomials and their properties}
\label{sec:MoserPoly}

Theorem \ref{thm:qskn_coeff} provides us with another proof of an important formula, which was originally obtained in 1962 by Gordon, Fraenkel, and Straus (\cite{GorFraStr}) specifically for the purpose of solving the Moser's Problem.

\begin{theorem}[Gordon-Fraenkel-Straus Theorem]
\label{thm:fskn}
 For any natural numbers $s$, $k$, $n$ such that $s, k \lte n$ we have
\begin{equation}
\label{eq:fsk_main}
\mskn skn = \sum_{p=1}^s (-1)^{p-1}\binom{n}{s-p}p^{k-1} \pp
\end{equation}
\end{theorem}

\begin{proof}
By definition, $\mskn skn = c_\lambda$, where $\lambda$ is the 1-part partition $\{k\}$. Thus $d=1$, $\lambda_1=d$, and $\delta_1=1$. There is only one length 1 composition $M$ of $s-p$, and therefore equality \eqref{eq:clambda} is reduced to the following
$$
c_\lambda = \sum_{p=0}^s (-1)^{s-p-1} \binom np (s-p)^{k-1} \pc
$$
which is equivalent to formula \eqref{eq:fsk_main}.
\end{proof}

Another way of deducing the last theorem from Theorem \ref{thm:qskn_coeff} is to use formulas \eqref{eq:lsp_binom} and \eqref{eq:c_ksp} for one specific $n$-multiset. 

\begin{notation}
\label{def:znk}
For natural numbers $n\gte k$ we define $\Znk$ as $n$-multiset that consists of $n-k$ zeros and all complex $k$th roots of unity; that is,
\begin{equation}
\label{eq:znk}
\Znk = \Big\{e^{2\pi i \cdot m/k}: m=0,1,\ldots, k-1 \Big\} \msu \Big\{ \underbrace{0,\ldots,0}_{n-k} \Big\} \pp
\end{equation}
\end{notation}

\begin{lemma}
\label{thm:fskn_znk} For $n\gte k$ and $n\gte s$ we have $\mskn skn = \mfp_k(\sms{\Znk}s)/k$.
\end{lemma}

\begin{proof} Obviously, if $k \mid j$, then $\mfp_j(\Znk) = k$, otherwise $\mfp_j(\Znk) = 0$. Therefore, in equation \eqref{eq:fsk_sigma} for this multiset the last summand on the right-hand side $\tqskn skn(\mfp_1(\Znk), \mfp_2(\Znk), \ldots, \mfp_{k-1}(\Znk))$ is zero. Hence, we have $\mfp_k\left(\sms{\Znk}s\right) = \mskn skn\, \mfp_k(\Znk)$.
\end{proof}

We will leave finalizing this slightly different approach to the reader.

\xxx

Formula \eqref{eq:fsk_main} shows us that $\mskn skn$ is a polynomial in $n$ of degree $s-1$, which leads us to the following.

\begin{definition}
For any natural numbers $s$ and $k$ we will define \textbf{Moser polynomial} $\fsk sk(x)$ by the formula
\begin{equation}
\label{eq:fsk_psum}
\fsk sk(x) = \sum_{j=1}^s (-1)^{j-1} j^{k-1} \binom{x}{s-j} \pp
\end{equation}
The \textbf{normalized Moser polynomial} $(s-1)!\fsk sk(x)$ has integer coefficients and will be denoted by $\nfsk sk(x)$.
\end{definition}

This means that for natural numbers $s$, $k$ and $n$ such that $n\gte s$ and $n\gte k$ we have
$$
\mskn skn = \fsk sk(n) \pp
$$

Formula \eqref{eq:fsk_psum} can be rewritten to explicitly show the Moser polynomial's coefficients:
$$
\fsk sk(x) = \sum_{j=0}^s x^j \left( (-1)^{s+j-1} \sum_{i=j}^s \frac{(s-i)^{k-1} }{i!}
\stirling{i}{j} \right) \pc
$$
where $\stirling{i}{j}$ denotes the unsigned (positive) Stirling number of the first kind (see Chapter 6 in \cite{GraKnuPat}). Thus for any Moser polynomial signs of its coefficients alternate.

\medskip

The following theorem (proved in \cite{GorFraStr}) follows directly from Theorem \ref{thm:fskn} combined with the definition of the Moser polynomials.

\begin{theorem}
\label{thm:moser_problem}
Given natural numbers $n$ and $s$ such that $n\gte s$, consider sequence $\fsk sk(n)$, $k = 1, \ldots, n$. If none of these values vanish, then the answer to the Moser problem is positive---in other words, any $n$-multiset $A$ can be uniquely recovered from multiset $\sms As$ of its $s$-sums.
\end{theorem}

\begin{proof}
To begin with, $\mfp_1(\sms As) = \mskn s1n \mfp_1(A)$. Since $\mskn s1n = \fsk s1(n) \neq 0$, $\mfp_1(A)$ is fully determined by $\mfp_1(\sms As)$, and, therefore, by the multiset $\sms As$.

Now easy induction by $k$, using formula \eqref{eq:fsk_sigma} and the fact that $\mskn skn = \fsk sk(n)$ is nonzero, shows that for any $k \lte n$ the power-sum $\mfp_k(A)$ is determined by the power-sums $\mfp_1(\sms As)$, \ldots, $\mfp_k(\sms As)$, and therefore, by the multiset $\sms As$. 

Finally, a multiset of $n$ numbers is fully determined by the sequence of its first $n$ power-sums, which concludes the proof.
\end{proof}

In this section we will prove several important properties of the Moser polynomials and their values.

\begin{prop}
\label{thm:fsk_nduality}
For any natural numbers $s$, $k > 1$, and $n\gte k$, the equality $\fsk sk(n) = (-1)^k\fsk {n-s}k(n)$ holds true.
\end{prop}

\begin{proof}
If $s \lte 0$ or $s\gte n$ then our equation follows from definition of Moser polynomials and from the previous item. Hence we can assume that $0 < s < n$.

Now let us again employ the $n$-multiset $A = \Znk$ from Lemma \ref{thm:fskn_znk}; we know that $\mfp_k(\sms As) = k \fsk sk(n)$.  

Let multiset $B$ be a reflection of $A$ in complex plane with respect to zero; in other words, $B = -A$. Then, obviously, $\sms A{n-s} = \sms Bs$ and $\mfp_k(B) = (-1)^k\mfp_k(A)$.

Thus we have
$$
k \fsk {n-s}k(n) \mfp_k(A) = \mfp_k(\sms A{n-s}) = \mfp_k(\sms Bs) = (-1)^k k \fsk sk(n) \pp
$$

Divide this equality by $k$ and we are done.
\end{proof}

\begin{prop} For any natural numbers $s$, $k$, and $n$ the following recurrency equations hold.
\label{thm:MoserPolyRecs}
\begin{enumerate}

\item\label{itm:mskn_recur_sx} $\vfs$
$\mskn sk{n+1} = \mskn skn + \mskn {s-1}kn$.

\item\label{itm:mskn_recur_ksx} $\vfs$
$\mskn s{k+1}n = s\cdot \mskn skn - n\cdot \mskn {s-1}k{n-1}$.

\end{enumerate} 
\end{prop}

\begin{proof} We begin with an easy but useful lemma.

\begin{lemma}
\label{thm:sk_tz}
Given $n$-multiset $A$ and number $z$ we construct $n$-multiset $A' = T_z(A)$, where the translation function $T_z:\mbr \rightarrow \mbr$ is defined by formula $T_z(x) = x+z$. Then the following equality holds true
$$
\mfp_k(A') = \sum_{i=0}^k \binom{k}{i} \mfp_{k-i}(A)\, z^i \pp
$$
\end{lemma}

\begin{proof}
Let us assume that $A = \{a_1, \ldots, a_n\}$. Then 
\begin{multline}
\mfp_k(A') = \sum_{j=1}^n (z+a_j)^k =
\sum_{j=1}^n\sum_{i=0}^k \binom{k}{i} z^i a_j^{k-i} = \\
= \sum_{i=0}^k \sum_{j=1}^n \binom{k}{i} z^i a_j^{k-i} = 
\sum_{i=0}^k \binom{k}{i} z^i \sum_{j=1}^n a_j^{k-i} = 
\sum_{i=0}^k \binom{k}{i} z^i \mfp_{k-i}(A) \nn
\end{multline}
\end{proof}

Item \ref{itm:mskn_recur_sx} immediately follows from formula \eqref{eq:fsk_psum}. However, for variety sake, we will present here a proof which relies only on definition of $\mskn skn$ given in \eqref{eq:fsk_sigma}.

First, the case $s\lte 1$ is obvious. Therefore we can assume that $s\gte 2$.

Second, we can assume that $n\gte k$ and $n\gte s$.

Now let us consider some $n$-multiset $A = \{a_1, \ldots, a_n\}$, arbitrary number $z$ and $n+1$-multiset $B = A \msu \{z\}$. Then multiset $\sms Bs$ is, obviously, equal to $\sms As \msu T_z(\sms A{s-1})$.

Let us consider all the expressions below as polynomials in $z$. For instance, power-sum $\mfp_k(\sms Bs)$ taken as such polynomial has degree $k$.

By definition of $\mskn skn$ we have
\begin{align}
\mfp_k(\sms Bs) &= \mskn sk{n+1} \mfp_k(B) +
\tqskn sk{n+1}(\mfp_1(B), \mfp_2(B), \ldots, \mfp_{k-1}(B)) \nc
\mfp_k(\sms{A}s) &= \mskn skn \mfp_k(A) +
\tqskn skn(\mfp_1(A), \mfp_2(A), \ldots, \mfp_{k-1}(A)) \nn
\end{align}
and we also know that
\begin{equation}
\label{eq:skbs_skas}
\mfp_k(\sms Bs) = \mfp_k(\sms{A}s) + \mfp_k(T_z(\sms A{s-1})) \pp
\end{equation}

Using lemma \ref{thm:sk_tz} we get
\begin{multline}
\mfp_k(T_z(\sms{A}{s-1})) = \sum_{i=0}^k \binom{k}{i} \mfp_{k-i}(\sms A{s-1})\, z^i = \\
   = \sum_{i=0}^k \binom{k}{i}\left( \mskn {s-1}{k-i}n \mfp_{k-i}(A) + \tqskn {s-1}{k-i}n(\mfp_1(A),\mfp_2(A),\ldots,\mfp_{k-i-1}(A)) \right)z^i \nn
\end{multline}

Now, let us compute the constant term of polynomials on both sides in \eqref{eq:skbs_skas}---for that we set $z=0$. With $z=0$ and $k>0$ we have $\mfp_k(B) = \mfp_k(A)$. Thus, the left-hand side of \eqref{eq:skbs_skas} evaluated at $z=0$ equals
$$
\mfp_k(\sms Bs) = \mskn sk{n+1} \mfp_k(A) +
\tqskn sk{n+1}(\mfp_1(A), \mfp_2(A), \ldots, \mfp_{k-1}(A))\pc
$$
and doing the same for the right-hand side of \eqref{eq:skbs_skas} we obtain
$$
\mskn skn \mfp_k(A) + \tqskn skn(\mfp_1(A), \mfp_2(A), \ldots, \mfp_{k-1}(A)) + 
\mskn {s-1}kn \mfp_{k}(A) + \tqskn{s-1}kn(\mfp_1(A),\mfp_2(A),\ldots,\mfp_{k-1}(A)) \pp \nn
$$
Coefficients at terms $\mfp_k$ in these two expressions must coincide, Q.E.D. We have just proved the required equality for infinitely many values of $n$---namely, for all sufficiently large natural numbers. Thus, we have also proved the following polynomial identity
\begin{equation}
\label{eq:fsk_recur_sx}
\fsk sk(x+1) = \fsk sk(x) + \fsk {s-1}k(x) \pp
\end{equation}

\smallskip

The most straightforward way to prove item~\ref{itm:mskn_recur_ksx} is to use Theorem \eqref{thm:fskn}. So instead of item~\ref{itm:mskn_recur_ksx} we have to prove polynomial identity
\begin{equation}
\label{eq:fsk_recur_ksx}
\fsk s{k+1}(x) = s \cdot \fsk sk(x) - x \cdot \fsk {s-1}k(x-1) \pp
\end{equation}
To demonstrate that some two polynomials are identical it is sufficient to prove that sequences of their coefficients are identical. Here instead of representing each one of these polynomials in the regular way---as a sum of power monomials $x^j$---and then comparing their coefficients, we will use another basis of the polynomial ring $\mbr[x]$, namely the basis of the ``falling power'' polynomials. Thus, we express each polynomial as a sum $\sum_{j=0}^m\beta_jx^{[j]}$ and then show that their ``falling-power'' coefficients $\beta_j$ coincide.

For the the left-hand side of our recurrence equation we have
$$
\fsk s{k+1}(x) = \sum_{j=0}^s (-1)^{j-1} j^{k} \binom{x}{s-j} = 
\sum_{j=0}^s x^{[s-j]} \frac{(-1)^{j-1} j^{k}}{(s-j)!}
$$
and for the right-hand side, using equality $a^{[b]} = a\cdot (a-1)^{[b-1]}$,
\begin{align}
s \cdot \fsk s{k}(x) - x \cdot \fsk {s-1}{k}(x-1) &= 
s \sum_{j=0}^s (-1)^{j-1} j^{k-1} \binom{x}{s-j} - x \sum_{j=0}^{s-1} (-1)^{j-1} j^{k-1} \binom{x-1}{s-1-j} =
\nc
&= \sum_{j=0}^s x^{[s-j]} \frac{(-1)^{j-1} s j^{k-1}}{(s-j)!} -
x \sum_{j=0}^{s-1} (x-1)^{[s-1-j]} \frac{(-1)^{j-1} j^{k-1}}{(s-1-j)!}
\nc
&= \sum_{j=0}^s x^{[s-j]}\frac{(-1)^{j-1} j^{k-1}}{(s-j)!}
\left( s - (s-j) \right) = \sum_{j=0}^s x^{[s-j]} \frac{(-1)^{j-1} j^{k}}{(s-j)!} \pp
\nn
\end{align}
\end{proof}

\medskip

Applying recurrency equation \eqref{eq:fsk_recur_sx} multiple times gives us the following.

\begin{prop}
\label{thm:fsk_recurrs_d}
For any number $x$ and any natural numbers $d$, $s$ and $k$ equalities
\begin{align}
\label{eq:fsk_recurr_d}
\fsk {s+d}k(x+d) &= \sum_{j=0}^d \binom{d}{j} \fsk {s+j}k(x) \nc
\fsk sk(x+d)     &= \fsk sk(x) + \sum_{j=0}^{d-1} \fsk {s-1}k(x+j) \pp \nn
\end{align}
hold true.
\end{prop}

Our next corollary shows how the formula for the Moser polynomials can be rewritten using the Eulerian numbers.

\begin{prop}
\label{thm:fsk_moser_eul}
\begin{equation}
\label{eq:fsk_moser_eul}
\fsk sk(x) = (-1)^{s-1}\sum_{j=0}^{s-1} (-1)^j \eulerian{k-1}{s-j-1} \binom{x-k}{j} \pp
\end{equation}
\end{prop}

\begin{proof}
This can be proved using the recurrency equations \ref{thm:MoserPolyRecs} but there is a more straightforward way. Regardless, we will need the explicit formula for Eulerian numbers (see \cite{GraKnuPat})
\begin{equation}
\label{eq:eulerian}
\eulerian{n}{m} = \sum_{j=0}^m(-1)^j\binom{n+1}j (m+1-j)^n
\end{equation}
(which, by the way, immediately allows us to see that $\fsk sk(k) = (-1)^{s-1}\eulerian{k-1}{s-1}$), and the well-known summation property of binomial coefficients
$$
\binom{x}{a} = \sum_{j=0}^{\infty} \binom{x-b}{j} \binom{b}{a-j} \pc
$$
which holds true for any integers $a$ and $b$.

Now,
\begin{align}
(-1)^{s-1}\fsk sk(x) &= \sum_{p=0}^s(-1)^p \binom{x}{p} (s-p)^{k-1} \nc
   &= \sum_{p=0}^s(-1)^p (s-p)^{k-1} \sum_{j=0}^{\infty} \binom{x-k}{j}\binom{k}{p-j} \nc
   &= \sum_{j=0}^{\infty} \binom{x-k}{j} \sum_{p=0}^s (-1)^p (s-p)^{k-1} \binom{k}{p-j} \nc
   &= \sum_{j=0}^{\infty} \binom{x-k}{j} \sum_{p=-j}^{s-j} (-1)^{p+j} (s-p-j)^{k-1} \binom{k}{p} \nc
   &= \sum_{j=0}^{\infty} \binom{x-k}{j} (-1)^{j} \sum_{p=0}^{s-j} (-1)^{p} \binom{k}{p} (s-p-j)^{k-1} \nc
   &= \sum_{j=0}^{\infty} \binom{x-k}{j} (-1)^j \eulerian{k-1}{s-j-1} \nc
   &= \sum_{j=0}^{s-1} (-1)^j \binom{x-k}{j} \eulerian{k-1}{s-j-1} \pp \nn
\end{align}
\end{proof}

When $x=n$ is an integer, this can be rewritten using the backward difference operator $\nabla$. Namely, in this case the formula \eqref{eq:fsk_moser_eul} is equivalent to
$$
\fsk sk(n) = (-1)^{s-1} \, \nabla^{n-k}_{s-1} \mca^{(k-1)} \pc
$$
where $\mca^{(k-1)}$ is the sequence of numbers $\eulerian{k-1}{m}$ constituting the $(k-1)$th row of the Eulerian triangle. That is, the Moser polynomial's value at $n$ is the $(n-k)$th discrete backward derivative (or corresponding discrete ``integral'', if $n<k$) of sequence $\mca^{(k-1)}$, computed at its $(s-1)$th term. Or, equivalently,
$$
(x-1)^{n-k}\,\mca_{k-1}(x)  = \sum_{s=0}^n (-1)^{s-1}\fsk sk(n)x^s
$$
where $\mca_{k-1}(x)$ is the $(k-1)$th Eulerian polynomial.

\smallskip

One interesting corollary of the previous Proposition and Lemma \ref{thm:fskn_znk}.

\begin{corollary}
\label{thm:zskn_eulerian}
Consider set $\Zkk$ (see \eqref{eq:znk}) which consists of all complex roots of unity of $k$th order. Then eulerian number $\eulerian{k-1}{s-1}$ equals the sum of $k$th powers of all $s$-sums of $\Zkk$ multiplied by $(-1)^{s-1}/k$. In other words,
$$
\eulerian{k-1}{s-1} = (-1)^{s-1} \, \mfp_k(\sms{\Zkk}{s}) / k \pp
$$
\end{corollary}
\begin{proof}
Since $k = \mfp_k(\Zkk)$ both these numbers are equal to $(-1)^{s-1} \fsk sk(k) $. 
\end{proof}

Finally, two more formulas. They express Moser polynomials using Stirling numbers of the second kind.

\begin{prop}
\label{thm:fsk_moser_stirling}
\begin{align}
\label{eq:fsk_via_stirling1}
\fsk sk(x) &= \sum_{i=1}^k (-1)^{i-1}(i-1)! \Stirling{k}{i} \binom{x-i}{s-i} \\
\label{eq:fsk_via_stirling2}
\fsk sk(x) &= \sum_{i=1}^k (-1)^{i+k-1} i! \Stirling{k-1}{i} \binom{x-i-1}{s-1} 
\end{align}
(obviously, all the summands with index $i > k$ are zeros so the upper summation limit could be, if necessary, changed to infinity; similarly, the lower summation limit can be changed to zero or even to $-\infty$.)
\end{prop}

\begin{proof}

We can assume without loss of generality that $x = n$ is a natural number. Then, using the explicit formula for the Stirling numbers of the second kind (see identity \textbf{6.19} in \cite{GraKnuPat})
$$
\Stirling{n}{m} = \frac 1{m!}\sum_{j=0}^m (-1)^{m-j}\binom{m}{j} j^n \pc
$$
we turn the right-hand side of \eqref{eq:fsk_via_stirling1} into
\begin{align}
\sum_{i=1}^k (-1)^{i-1}(i-1)! \Stirling{k}{i} \binom{n-i}{n-s}
  &= 
\sum_{i=1}^s \sum_{j=1}^s (-1)^{i-1}(i-1)! \frac 1{i!} (-1)^{i-j} \binom ij j^k \binom{n-i}{n-s}
\nc
  &=
\sum_{i=1}^s \sum_{j=1}^s (-1)^{j-1} j^{k-1} \binom{n-i}{n-s} \binom{i-1}{j-1}
\nc
  &=
\sum_{j=1}^s (-1)^{j-1} j^{k-1} \sum_{i=1}^s \binom{i-1}{j-1} \binom{n-i}{n-s}
\nn
\end{align}
(notice that the limits' adjustments done here do not affect the sums.) Now we see that it is sufficient to prove that for any $s$, $j$, and $n$ the equality
\begin{equation}
\label{eq:ij_nis}
\sum_{i=0}^s \binom{i}{j} \binom{n-i}{n-s} = \binom{n+1}{s-j}
\end{equation}
holds true. To do that, consider the generating function
$$
\frac{t^j}{(1-t)^{j+1}} = \sum_{i=0}^{\infty} \binom{i}{j} t^i 
$$
and compare coefficients at $t^n$ on the both sides of the equality
$$
\frac{t^j}{(1-t)^{j+1}} \cdot \frac{t^{n-s}}{(1-t)^{n-s+1}} = \frac{t^{n-s+j}}{(1-t)^{n-s+j+2}} = 
\frac 1t \cdot \frac{t^{n-s+j}}{(1-t)^{n-s+j+1}} \pp
$$
Note: Formula \eqref{eq:ij_nis} is basically the same as a slightly more general identity \textbf{5.26} in \cite{GraKnuPat}, which can be proved in the exactly same manner.

To prove \eqref{eq:fsk_via_stirling2}, we will start with equality
$$
\eulerian{n}{k} = \sum_{i=k}^{n-1} (-1)^{i-k} \binom{i}{k}(n-i)! \Stirling{n}{n-i}
$$
proved in \cite{Kno}. Substituting that into \eqref{eq:fsk_moser_eul} we obtain
\begin{align}
\fsk sk(n) &= (-1)^{s-1}\sum_{j=0}^{s-1} (-1)^j \eulerian{k-1}{s-j-1} \binom{x-k}{j} \nc
&= (-1)^{s-1}\sum_{j=0}^{s-1} \sum_{i=s-j-1}^{k-2} (-1)^j \binom{n-k}{j} (-1)^{i-(s-j-1)} \binom{i}{s-j-1}(k-1-i)! \Stirling{k-1}{k-1-i} \nc
&= (-1)^{s-1}\sum_{i=0}^{k-1} (-1)^{i-s+1} (k-1-i)! \Stirling{k-1}{k-1-i} \sum_{j=0}^{s-1} \binom{n-k}{j} \binom{i}{s-j-1} \nc
&= \sum_{i=0}^{k-1} (-1)^i (k-1-i)! \Stirling{k-1}{k-1-i} \binom{n-k+i}{s-1}  \pp
\nn
\end{align}
and making substitution $i \rightarrow k-1-i$, we have formula \eqref{eq:fsk_via_stirling2} as well.

\end{proof}

\presection

\end{document}